\documentclass[reqno]{amsart}
\usepackage{amssymb}
\usepackage{hyperref}
\usepackage{amssymb,tikz}

\newcommand{\nn}{\mathbb N}
\newcommand{\ee}{\mathbb E}
\newcommand{\rr}{\mathbb R}

\newtheorem{theorem}{Theorem}[section]
\newtheorem*{lemma}{Lemma}

\begin{document}
\title[Strassen's Algorithm for Matrices of Arbitrary Size]
{Strassen's Matrix Multiplication Algorithm for Matrices of Arbitrary Order}

\author[I. Hedtke]{Ivo Hedtke}

\address{Mathematical Institute, University of Jena, D-07737 Jena, Germany}
\email{Ivo.Hedtke@uni-jena.de}
\thanks{The author was supported by the Studienstiftung des Deutschen
Volkes.}

\subjclass[2010]{Primary 65F30; Secondary 68Q17}
\keywords{Fast Matrix Multiplication, Strassen Algorithm}

\begin{abstract}
The well known algorithm of \textsc{Volker Strassen} for matrix multiplication
can only be used for $(m2^k \times m2^k)$ matrices. For arbitrary $(n \times
n)$ matrices one has to add zero rows and columns to the given matrices to use
\textsc{Strassen}'s algorithm. \textsc{Strassen} gave a strategy of how to set
$m$ and $k$ for arbitrary $n$ to ensure $n\leq m2^k$. In this paper we study the
number $d$ of additional zero rows and columns and the influence on the number
of flops used by the algorithm in the worst case ($d=n/16$), best case ($d=1$)
and in the average case ($d\approx n/48$). The aim of this work is to give a
detailed analysis of the number of additional zero rows and columns and the
additional work caused by \textsc{Strassen}'s bad parameters. 
\textsc{Strassen} used the parameters $m$ and $k$ to
show that his matrix multiplication algorithm needs less than $4.7n^{\log_2 7}$
flops. We can show in this paper, that these parameters cause an additional
work of approximately 20 \%  in the worst case in comparison to the optimal strategy
for the worst case. This is the main reason for the search for better
parameters.
\end{abstract}

\maketitle
\numberwithin{equation}{section}

\section{Introduction}
\noindent In his paper ``\emph{Gaussian Elimination is not Optimal}''
(\cite{StrassenGauss}) \textsc{Volker Strassen} developed a recursive algorithm
(we will call it $\mathcal S$) for multiplication of square matrices of
order $m2^k$. The algorithm itself is described below. Further details can be
found in \cite[p.~31]{Golub}.

Before we start with our analysis of the parameters of
\textsc{Strassen}'s algorithm we will have a short look on the history of fast matrix
multiplication. The naive algorithm for matrix multiplication is an $\mathcal 
O(n^3)$ algorithm. In $1969$ \textsc{Strassen} showed that there is an $\mathcal
O(n^{2.81})$ algorithm for this problem. \textsc{Shmuel Winograd} optimized
\textsc{Strassen}'s algorithm. While the \textsc{Strassen-Winograd} algorithm is
a variant that is always implemented (for example in the famous GEMMW package),
there are faster ones (in theory) that are impractical to implement. The fastest
known algorithm, devised in 1987 by \textsc{Don Coppersmith} and
\textsc{Winograd}, runs in $\mathcal O(n^{2.38})$ time. There is also an
interesting group-theoretic approach to fast matrix multiplication from
\textsc{Henry Cohn} and \textsc{Christopher Umans}, see \cite{CohnEins},
\cite{CohnZwei} and \cite{Robinson}. Most researchers believe that an optimal
algorithm with $\mathcal O(n^2)$ runtime exists, since then no further
progress was made in finding one.

Because modern architectures have complex memory hierarchies and increasing
parallelism, performance has become a complex tradeoff, not just a simple matter
of counting flops (in this article one flop means one floating-point
operation, that means one addition is a flop and one multiplication is one flop,
too). Algorithms which make use of this technology were described by
\textsc{Paolo D'Alberto} and \textsc{Alexandru Nicolau} in \cite{Alberto}. An
also well known method is \emph{Tiling}: The normal algorithm can be speeded up
by a factor of two by using a six loop implementation that blocks submatrices so
that the data passes through the L1 Cache only once.

\subsection{The algorithm}
Let $A$ and $B$ be ($m2^k \times m2^k$) matrices. To compute $C:=AB$ let
\begin{gather*}
A=\begin{bmatrix}
A_{11} & A_{12}\\
A_{21} & A_{22}
\end{bmatrix}\quad\text{and}\quad
B=\begin{bmatrix}
B_{11} & B_{12}\\
B_{21} & B_{22}
\end{bmatrix},
\end{gather*}
where $A_{ij}$ and $B_{ij}$ are matrices of order $m2^{k-1}$. With the following
auxiliary matrices
\begin{xalignat*}{2}
H_1 &:= (A_{11} + A_{22})(B_{11} + B_{22}) &
H_2 &:=(A_{21} + A_{22})B_{11} \\
H_3 &:=A_{11}(B_{12} - B_{22}) &
H_4 &:= A_{22}(B_{21} - B_{11}) \\
H_5 &:=(A_{11} + A_{12})B_{22} &
H_6 &:=(A_{21} - A_{11})(B_{11} + B_{12})\\
H_7 &:= (A_{12} - A_{22})(B_{21} + B_{22})
\end{xalignat*}
we get
\begin{gather*}
C = \begin{bmatrix}
H_1 + H_4 - H_5 + H_7 & H_3 + H_5\\
H_2 + H_4 & H_1 + H_3 - H_2 + H_6
\end{bmatrix}.
\end{gather*}
This leads to recursive computation. In the last step of the recursion the
products of the $(m \times m$) matrices
are computed with the naive algorithm (straight forward implementation
with three \texttt{for}-loops, we will call it $\mathcal N$).

\subsection{Properties of the algorithm}\label{Property}
The algorithm $\mathcal S$ needs (see \cite{StrassenGauss})
\begin{gather*}
F_{\mathcal S}(m,k) := 7^km^2(2m+5) - 4^k 6 m^2
\end{gather*}
flops to
compute the
product of two square matrices of order $m2^k$. The naive algorithm $\mathcal
N$
needs $n^3$ multiplications and $n^3 - n^2$ additions to compute the product of
two ($n \times n$) matrices. Therefore $F_{\mathcal N} (n) = 2n^3 - n^2$. In the
case $n=2^p$ the algorithm $\mathcal S$ is better than $\mathcal N$ if
$F_{\mathcal S}(1,p) < F_{\mathcal N}(2^p)$, which is the case iff $p \geqslant
10$. But if we use algorithm $\mathcal S$ only for matrices of order at least
$2^{10}=1024$, we get a new problem:

\begin{lemma}
The algorithm $\mathcal S$ needs $\frac{17}{3}(7^p - 4^p)$ units of memory (we
write ``uom'' in short) (number of \texttt{float}s or
\texttt{double}s) to compute
the
product of two $(2^p \times 2^p)$ matrices.
\end{lemma}

\begin{proof}
Let $M(n)$ be the number of uom used by $\mathcal S$ to compute the product of
matrices of order $n$. The matrices $A_{ij}$, $B_{ij}$ and $H_\ell$ need $15
(n/2)^2$ uom. During the computation of the auxiliary matrices $H_\ell$ we need
$7 M(n/2)$ uom and $2(n/2)^2$ uom as input arguments for the recursive calls
of $\mathcal S$. Therefore we get $M(n) = 7M(n/2) + (17/4)n^2$. Together with
$M(1)=0$ this yields to $M(2^p) = \frac{17}{3}(7^p - 4^p)$.
\end{proof}

As an example, if we compute the product of two $(2^{10} \times 2^{10})$
matrices (represented as \texttt{double} arrays) with $\mathcal S$ we need
$8\cdot \frac{17}{3}(7^{10} -
4^{10})$ bytes, i.e.~$12.76$ gigabytes of memory. That is an enormous amount of
RAM for such a problem instance. \textsc{Brice Boyer} et
al.~(\cite{Boyer}) solved this problem
 with fully in-place schedules of \textsc{Strassen-Winograd}'s
algorithm (see the following paragraph), if the input matrices can be
overwritten.

\textsc{Shmuel Winograd} optimized \textsc{Strassen}'s algorithm. The
\textsc{Strassen-Wino\-grad} algorithm (described in \cite{Probert}) needs only
$15$ additions and subtractions, whereas $\mathcal S$ needs $18$.
\textsc{Winograd} had also shown (see \cite{Winograd}), that the minimum number
of multiplications required to multiply $2 \times 2$ matrices is $7$.
Furthermore, \textsc{Robert Probert} (\cite{ProbertAdd}) showed that
$15$ additive operations are necessary and sufficient to multiply two $2 \times
2$ matrices with $7$ multiplications.

Because of the bad properties of $\mathcal S$ with full recursion and large
matrices, one can study the idea to use only one step of recursion. If $n$ is
even and we use one step of recursion of $\mathcal S$ (for the remaining
products we use $\mathcal N$) the ratio of this operation count to that
required by $\mathcal N$ is (see \cite{Huss})
\begin{gather*}
\frac{7n^3 + 11n^2}{8n^3-4n^2} \quad \xrightarrow{n\to\infty} \quad \frac78.
\end{gather*}
Therefore the multiplication of two sufficiently large matrices using
$\mathcal S$ costs approximately $12.5$ \% less than using $\mathcal N$.

Using the technique of stopping the recursion in the \textsc{Strassen-Winograd} algorithm early, there are
well known implementations, as for example
\begin{itemize}
 \item on the Cray-2 from \textsc{David Bailey} (\cite{Bailey}),
 \item GEMMW from \textsc{Douglas} et al. (\cite{Douglas}) and
 \item a routine in the IBM ESSL library routine (\cite{IBM}).
\end{itemize}

\subsection{The aim of this work}
\textsc{Strassen}'s algorithm can only be used for $(m2^k \times m2^k)$
matrices. For arbitrary $(n \times n)$ matrices one has to add zero rows and
columns to the given matrices (see the next section) to use \textsc{Strassen}'s
algorithm. \textsc{Strassen} gave a strategy of how to set $m$ and $k$ for
arbitrary $n$ to ensure $n\leq m2^k$. In this paper we study the number $d$ of
additional zero rows and columns and the influence on the number of flops used
by the algorithm in the worst case, best case and in the average case.\par
It is known (\cite{Probert}), that these parameters are not optimal. We only
study the number $d$ and the additional work caused by the bad
parameters of \textsc{Strassen}. We give no better strategy of how to set $m$
and $k$, and we do not analyze other strategies than the one from
\textsc{Strassen}.

\section{Strassen's parameter for matrices of arbitrary order}

\noindent Algorithm $\mathcal S$ uses recursions to multiply matrices of order
$m2^k$. If $k=0$ then $\mathcal S$ coincides with the naive algorithm $\mathcal
N$. So we will only consider the case where $k>0$. To use $\mathcal S$ for
arbitrary $(n\times n)$ matrices $A$ and $B$ (that
means for arbitrary $n$) we have to embed them into matrices $\tilde A$ and
$\tilde B$ which are both $(\tilde n \times \tilde n)$ matrices with $\tilde n
:= m2^k \geqslant n$. We do this by adding $\ell:=\tilde n - n$ zero rows and
colums to $A$ and $B$. This results in
\begin{gather*}
\tilde A\tilde B =
\begin{bmatrix}
A & 0^{n\times \ell}\\
0^{\ell\times n} & 0^{\ell\times \ell}
\end{bmatrix}
\begin{bmatrix}
B & 0^{n\times \ell}\\
0^{\ell\times n} & 0^{\ell\times \ell}
\end{bmatrix}
=:\tilde C,
\end{gather*}
where $0^{k\times j}$ denotes the $(k\times j)$ zero matrix. If we delete the
last $\ell$ columns and rows of $\tilde C$ we get the result $C=AB$.

We now focus on how to find $m$ and $k$ for arbitrary $n$ with $n
\leqslant m2^k$. An optimal but purely theoretical choice is
\begin{gather*}
(m^*,k^*) = \arg \min
\{F_{\mathcal S}(m,k): (m,k)\in\nn\times\nn_0, n\leqslant
m2^k\}.
\end{gather*}
Further methods of finding $m$ and $k$ can be
found in \cite{Probert}. We choose another way. 
According to \textsc{Strassen}'s proof of the main result of
\cite{StrassenGauss}, we define
\begin{gather}\label{Param}
k := \lfloor \log_2 n \rfloor - 4 \qquad \text{and} \qquad m:= \lfloor n2^{-k}
\rfloor + 1,
\end{gather}
where $\lfloor x \rfloor$ denotes the largest integer not greater than $x$. We
define $\tilde n := m2^k$ and study the relationship between $n$ and $\tilde
n$. The results are:
\begin{itemize}
	\item \emph{worst case}: $\tilde n \leqslant (17/16)n$,
\item \emph{best case}: $\tilde n \geqslant n+1$ and
\item \emph{average case}: $\tilde n \approx (49/48)n$.
\end{itemize}~

\subsection{Worst case analysis}

\begin{theorem}\label{THM:worst}
Let $n\in\nn$ with $n\geqslant 16$. For the parameters \eqref{Param} and
$m2^k=\tilde n$ we have
\begin{gather*}
\tilde n \leq \frac{17}{16} n.
\end{gather*}
If $n$ is a power of two, we have $\tilde n = \tfrac{17}{16}n$.
\end{theorem}

\begin{proof}
For fixed $n$ there is exactly one $\alpha \in \nn$ with $2^\alpha \leq n <
2^{\alpha + 1}$.
We define $I^\alpha := \{2^\alpha, \ldots, 2^{\alpha +1}-1\}$. Because of
\eqref{Param} for each $n\in I^\alpha$ the value of $k$ is
\begin{gather*}
k = \lfloor \log_2 n \rfloor - 4 = \log_2 2^\alpha - 4 = \alpha - 4.
\end{gather*}
Possible values for $m$ are
\begin{gather*}
m = \left\lfloor n\frac1{2^{\alpha - 4}} \right\rfloor +1
= \left\lfloor n\frac{16}{2^\alpha} \right\rfloor + 1 =: m(n).
\end{gather*}
$m(n)$ is increasing in $n$ and $m(2^{\alpha})=17$ and $m(2^{\alpha+1})=33$.
Therefore we have $m\in \{17,\ldots,32\}$. For each $n\in I^\alpha$ one of the
following inequalities holds:
\begin{gather*}
\begin{array}{lcrcccl}
(I_1^\alpha) && 2^\alpha = 16 \cdot 2^{\alpha-4} & \leq & n & < & 17\cdot
2^{\alpha-4}\\
(I_2^\alpha) && 17\cdot 2^{\alpha-4} & \leq & n & < & 18\cdot 2^{\alpha-4} \\
&&&& \vdots\\
(I_{16}^\alpha) && 31\cdot 2^{\alpha-4} & \leq & n & < & 32\cdot 2^{\alpha-4} =
2^{\alpha + 1}.
\end{array}
\end{gather*}
Note that $I^\alpha= \biguplus_{j=1}^{16}I_j^\alpha$. It follows, that
for all $n \in \nn$ there exists exactly one $\alpha$ with $n \in I^\alpha$ and
for all $n\in I^\alpha$ there is exactly one $j$ with $n\in I^\alpha_j$.

Note that for all $n\in I^\alpha_j$ we have $k=\alpha-4$ and $m(n) = j+16$. If
we only focus on $I^\alpha_j$ the difference $\tilde n - n$ has its maximum at
the lower end of $I^\alpha_j$ ($\tilde n$ is constant and $n$ has its minimum
at the lower end of $I^\alpha_j$). On $I^\alpha_j$ the value of $\tilde n$ and
the minimum of $n$ are
\begin{gather*}
  \tilde n_j^\alpha := (16+j)\cdot 2^{\alpha - 4}
  \qquad \text{and} \qquad
  n_j^\alpha := (15 + j) \cdot 2^{\alpha - 4}.
\end{gather*}
Therefore the difference $d_j^\alpha := \tilde n_j^\alpha - n_j^\alpha$ is
constant:
\begin{gather*}
  d_j^\alpha = (16+j) \cdot 2^{\alpha - 4} - (15+j) \cdot 2^{\alpha - 4} =
2^{\alpha - 4}
  \quad \text{for all } j.
\end{gather*}
To set this in relation with $n$ we study
\begin{gather*}
  r_j^\alpha := \frac{\tilde n_j^\alpha}{n_j^\alpha}
  = \frac{n_j^\alpha + d_j^\alpha}{n_j^\alpha} = 1 +
  \frac{d_j^\alpha}{n_j^\alpha}
  = 1 + \frac{2^{\alpha - 4} }{n_j^\alpha}.
\end{gather*}
Finally $r_j^\alpha$ is maximal, iff $n_j^\alpha$ is minimal, which is the case
for $n_1^\alpha = 16\cdot 2^{\alpha - 4}=2^\alpha$. With  $r_1^\alpha = 17/16$
we get $\tilde n \leq \frac{17}{16}n$, which completes the proof.
\end{proof}

Now we want to use the result above to take a look at the number of flops we
need for
$\mathcal S$ in the worst case. The worst case is $n=2^p$ for any $4 \leq p \in
\nn$. An optimal decomposition (in the sense of minimizing the number ($\tilde n
- n$) of zero rows and columns we add to the given matrices) is $m=2^j$ and
$k=p-j$, because $m2^k=2^j2^{p-j}=2^p=n$. Note that these parameters $m$ and $k$
have nothing to do with equation \eqref{Param}. Lets have a look on the
influence of $j$:

\begin{lemma}
Let $n=2^p$. In the decomposition $n=m2^k$ we use $m=2^j$ and $k=p-j$. Then
$f(j):=F_{\mathcal S}(2^j,p-j)$ has its minimum at $j=3$.
\end{lemma}

\begin{proof}
We have $f(j)=2\cdot 7^p (8/7)^j + 5\cdot 7^p (4/7)^j - 4^p 6$.
Thus
\begin{align*}
&f(j+1)-f(j)\\
&= [2\cdot 7^p (8/7)^{j+1} + 5\cdot 7^p (4/7)^{j+1} - 4^p 6] - [2\cdot 7^p
(8/7)^j + 5\cdot 7^p (4/7)^j - 4^p 6]\\
&= 2 \cdot 7^p (8/7)^j( 8/7 -1 ) + 5 \cdot 7^p (4/7)^j(
4/7 -1 )\\
&= 2 \cdot 7^p (8/7)^j \cdot 1/7 - 15 \cdot 7^p (4/7)^j \cdot 1/7 = 2 (4/7)^j
7^{p-1} (2^j-7.5).
\end{align*}
Therefore, $f(j)$ is a minimum if $j = \min\{i: 2^i-7.5 >0\} =
3$.
\end{proof}
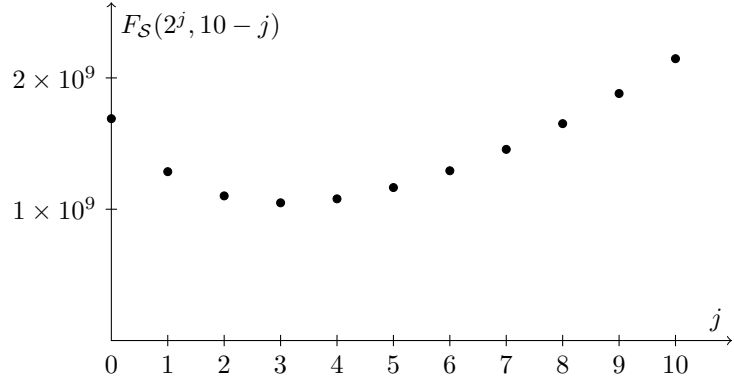
\begin{figure}
\centering
\begin{tikzpicture}[scale=0.75]
\draw[<->] (0,6) node[below right] {$F_{\mathcal S}(2^j,10-j)$} -- (0,0) --
(11,0) node[above left] {$j$}; \draw (0,-.1) node[below] {$0$};
\foreach \x in {1,...,10} {
\draw (\x,.1) -- (\x,-.1) node[below] {$\x$};
}
\draw plot[only marks, mark=*] coordinates {
(10,    5.0000)
(9.0000,    4.3820)
(8.0000,    3.8477)
(7.0000,    3.3914)
(6.0000,    3.0118)
(5.0000,    2.7139)
(4.0000,    2.5135)
(3.0000,    2.4433)
(2.0000,    2.5661)
(1.0000,    2.9958)
(0.0000,    3.9358)
};
\draw (.1,2.3294) -- (-.1,2.3294) node[left] {$1\times 10^9$};
\draw (.1,2*2.3294) -- (-.1,2*2.3294) node[left] {$2\times 10^9$};
\end{tikzpicture}
\caption{Different parameters ($m=2^j$, $k=10-j$, $n=m2^k$) to apply in the
\textsc{Strassen} algorithm for matrices of order $2^{10}$.}
\label{fig:ExampleStrassenParam}
\end{figure}
Figure \ref{fig:ExampleStrassenParam} shows that it is not optimal to use
$\mathcal S$ with full recursion in the example $p=10$. Now we study the worst
case $n=2^p$ and different sets of parameters
$m$ and $k$ for $F_{\mathcal S}(m,k)$:
\begin{enumerate}
	\item If we use equation \eqref{Param}, we get the original parameters
of
\textsc{Strassen}: $k=p-4$ and $m=17$. Therefore we define
\begin{gather*}
F_1(p) := F_{\mathcal S}(17,p-4) = 7^p \frac{39 \cdot 17^2}{7^4} -
4^p\frac{6\cdot 17^2}{4^4}.
\end{gather*}
\item Obviously $m=16$ would be a better choice, because with this we get
$m2^k=n$ (we avoid the additional zero rows and columns). Now we define
\begin{gather*}
F_2(p) := F_{\mathcal S}(16,p-4) = 7^p \frac{37 \cdot 16^2}{7^4} - 4^p6.
\end{gather*}
\item Finally we use the lemma above. With $m=8=2^3$ and $k=p-3$ we get
\begin{gather*}
F_3(p) := F_{\mathcal S}(8,p-3) = 7^p \frac{3 \cdot 64}{49} - 4^p6.
\end{gather*}
\end{enumerate}

Now we analyze $F_i$ relative to each other. Therefore we define
$f_j:=F_j/F_3$ ($j=1,2$). So we have $f_j\colon \{4,5,\ldots\} \to
\rr$, which is monotonously decreasing in $p$. With
\begin{gather*}
f_1(p) = \frac{289(4^p\cdot 2401 - 7^p\cdot 1664)}{12544(4^p\cdot 49 -
7^p\cdot 32)} \qquad
\text{and}\qquad f_2(p) = \frac{4^p\cdot7203 - 7^p\cdot4736}{147(4^p\cdot49 -
7^p\cdot32)}
\end{gather*}
we get
\begin{align*}
f_1(4) &= 3179/2624\approx 1.2115 &\lim_{p\to\infty} f_1(p)
&=3757/3136 \approx 1.1980\\
f_2(4) &= 124/123 \approx 1.00813 & \lim_{p\to\infty} f_2(p)
&=148/147 \approx 1.00680.
\end{align*}
\begin{figure}
  \centering
  \begin{tikzpicture}[scale=0.75]
    \draw[->] (0,0.75) -- (0,4) node[below right] {$f_1(p)$};
    \draw[dashed] (0,0.25) -- (0,0.75);
    \draw[->] (0,0) -- (6,0) node[above left] {$p$};
	\foreach \x in {4,6,...,20} {
		\draw (\x/3-4/3,.1) -- (\x/3-4/3,-.1) node[below] {$\x$};
	}
     \draw plot[only marks, mark=*] coordinates {
	(0,         3.1509)
	(0.3333,    2.5135)
	(0.6667,    2.1915)
	(1.0000,    2.0195)
	(1.3333,    1.9247)
	(1.6667,    1.8717)
	(2.0000,    1.8418)
	(2.3333,    1.8248)
	(2.6667,    1.8151)
	(3.0000,    1.8096)
	(3.3333,    1.8065)
	(3.6667,    1.8047)
	(4.0000,    1.8037)
	(4.3333,    1.8031)
	(4.6667,    1.8027)
	(5.0000,    1.8026)
	(5.3333,    1.8024)
    };
    \draw[dashed] (0,1.80229592) -- (5.3333,1.80229592);
    \draw (.1,1) -- (-.1,1) node[left] {$1.19$};
    \draw (.1,2) -- (-.1,2) node[left] {$1.20$};
    \draw (.1,3) -- (-.1,3) node[left] {$1.21$};
  \end{tikzpicture} ~~~~ 
  \begin{tikzpicture}[scale=0.75]
    \draw[->] (0,0.75) -- (0,4) node[below right] {$f_2(p)$};
    \draw[dashed] (0,0.25) -- (0,0.75);
      \draw[->] (0,0) -- (6,0) node[above left] {$p$};
  	\foreach \x in {4,6,...,20} {
  		\draw (\x/3-4/3,.1) -- (\x/3-4/3,-.1) node[below] {$\x$};
  	}
       \draw plot[only marks, mark=*] coordinates {
	(     0,    3.1301)
	(0.3333,    2.5027)
	(0.6667,    2.1858)
	(1.0000,    2.0165)
         (1.3333,    1.9232)
	(1.6667,    1.8711)
	(2.0000,    1.8416)
	(2.3333,    1.8249)
	(2.6667,    1.8154)
	(3.0000,    1.8099)
	(3.3333,    1.8068)
	(3.6667,    1.8051)
	(4.0000,    1.8041)
	(4.3333,    1.8035)
	(4.6667,    1.8032)
	(5.0000,    1.8030)
	(5.3333,    1.8029)
    };
    \draw[dashed] (0,1.80272109) -- (5.3333,1.80272109);
      \draw (.1,1) -- (-.1,1) node[left] {$1.006$};
      \draw (.1,2) -- (-.1,2) node[left] {$1.007$};
      \draw (.1,3) -- (-.1,3) node[left] {$1.008$};
    \end{tikzpicture}
  \caption{Comparison of \textsc{Strassen}'s parameters ($f_1$, $m=17$,
  $k=p-4$), obviously better parameters ($f_2$, $m=16$, $k=p-4$) and the
  optimal parameters (lemma, $m=8$, $k=p-3$) for the worst case $n=2^p$.
Limits of $f_j$ are dashed.}
  \label{fig:ExampleWorstRelative}
\end{figure}
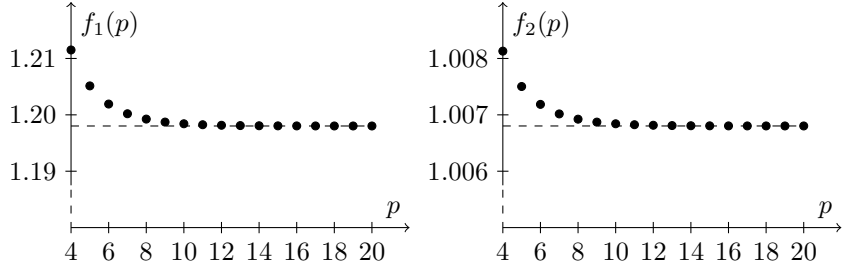

Figure \ref{fig:ExampleWorstRelative} shows the graphs of the functions $f_j$.
In conclusion, in the worst case the parameters of
\textsc{Strassen} need approx.~20 \% more flops than the optimal parameters of
the lemma.

\subsection{Best case analysis}

\begin{theorem}
Let $n\in\nn$ with $n\geqslant 16$. For the parameters \eqref{Param} and
$m2^k=\tilde n$ we have
\begin{gather*}
n+1 \leq \tilde n.
\end{gather*}
If $n=2^p\ell-1$, $\ell\in\{16,\ldots,31\}$, we have $\tilde n = n+1$.
\end{theorem}

\begin{proof}
Like in Theorem \ref{THM:worst} we have for each $I_j^\alpha$
a constant value for $\tilde n_j^\alpha$ namely $2^{\alpha - 4}(16+j)$.
Therefore $n < \tilde n$ holds. The difference $\tilde n -n$ has its minimum at
the
upper end of $I^\alpha_j$. There we have $\tilde n - n = 2^{\alpha - 4}(16+j) -
(2^{\alpha - 4}(16+j) - 1)=1$. This shows $n+1 \leq \tilde n$.
\end{proof}

Let us focus on the flops we need for $\mathcal S$, again. Lets have a look at
the example $n=2^p-1$. The original parameters (see equation \eqref{Param})
for $\mathcal S$ are $k=p-5$ and $m=32$. Accordingly we define
$F(p):=F_{\mathcal S}(32,p-5)$. Because $2^p - 1\leq 2^p$ we can add $1$ zero
row and column and use the lemma from the worst case. Now we get the parameters
$m=8$ and $k=p-3$ and define $\tilde F(p) := F_{\mathcal S}(8,p-3)$. To analyze
$F$ and $\tilde F$ relative to each other we have
 \begin{gather*}
 r(p) := \frac{F(p)}{\tilde F(p)}= \frac{4^p\cdot 16807 - 7^p\cdot
11776}{4^p\cdot 16807 - 7^p\cdot
 10976}.
\end{gather*}
Note that $r\colon \{5,6,\ldots\} \to \rr$ is monotonously decreasing in $p$ and
has its
maximum at $p=5$. We get
 \begin{align*}
 r(5) &= 336/311\approx 1.08039 & \lim_{p\to\infty} r(p) =
11776/10976  \approx 1.07289.
 \end{align*}

Therefore we can say: In the best case the parameters of \textsc{Strassen} are
approx. $8$~\% worse than the optimal parameters from the lemma in the worst
case.

\subsection{Average case analysis}
With $\ee \tilde n$ we denote the expected value of $\tilde n$. We search for a
relationship like $\tilde n \approx \gamma n$ for $\gamma \in \rr$. That means
$\ee[\tilde n / n] = \gamma$.

\begin{theorem}
For the parameters \eqref{Param} of
\textsc{Strassen} $m2^k=\tilde n$ we have
\begin{gather*}
\ee \tilde n = \frac{49}{48}n.
\end{gather*}
\end{theorem}

\begin{proof}
First we focus only on $I^\alpha$. We write $\ee^\alpha := \ee \vert_{I^\alpha}$
and $\ee^\alpha_j := \ee \vert_{I^\alpha_j}$ for the expected value on
$I^\alpha$ and $I^\alpha_j$, resp. We have
\begin{gather*}
\ee^\alpha \tilde n	= \frac1{16}\sum_{j=1}^{16} \ee_j^\alpha\tilde n
= \frac1{16}\sum_{j=1}^{16} \tilde n_j^\alpha
= \frac1{16}\sum_{j=1}^{16} (j+16)2^{\alpha - 4}
= 2^{\alpha - 5}\cdot49.
\end{gather*}
Together with
$\ee^\alpha n = \frac12[(2^{\alpha + 1} - 1) + 2^\alpha]=2^\alpha + 2^{\alpha -
1} - 1/2$, we get
\begin{gather*}
\rho(\alpha) 	:= \ee^\alpha \left[\frac{\tilde n}{n}\right]
= \frac{\ee^\alpha \tilde n}{\ee^\alpha n}
= \frac{2^{\alpha - 5}\cdot49}{2^\alpha + 2^{\alpha - 1} - 1/2}.
\end{gather*}
Now we want to calculate $\ee_k := \ee\vert_{U(k)}[\tilde n / n]$, where
$U(k):=\biguplus_{j=0}^k I^{4+j}$ by using the values $\rho(j)$. Because of
$|I^5|=2|I^4|$ and $|I^4\cup I^5|=3|I^4|$ we have
$\ee_1 = \frac13 \rho(4)+ \frac23 \rho(5)$.
With the same argument we get
\begin{gather*}
\ee_k = \sum_{j=4}^{4+k} \beta_j \rho(j) \qquad \text {where} \qquad \beta_j =
\frac{2^{j-4}}{2^{k+1}-1}.
\end{gather*}
Finally we have
\begin{align*}
\ee \left[\frac{\tilde n}{n}\right] &= \lim_{k\to\infty} \ee_k
= \lim_{k\to\infty} \Bigg( \sum_{j=4}^{4+k}\beta_j \rho(j) \Bigg)\\
&= \lim_{k\to\infty} \Bigg( \frac{49}{2^{k+1}-1} \sum_{j=4}^{4+k} \frac{2^{2j -
9}}{2^j + 2^{j - 1} - 1/2}\Bigg)
= \frac{49}{48},
\end{align*}
what we intended to show.
\end{proof}
Compared to the worst case ($\tilde n \leq \frac{17}{16}n$, $17/16 = 1 + 1/16$),
note that $49/48 = 1+1/48 = 1+\frac{1}{3}\cdot \frac{1}{16}$.

\section{Conclusion}
\noindent\textsc{Strassen} used the parameters $m$ and $k$ in the form
\eqref{Param} to
show that his matrix multiplication algorithm needs less than $4.7n^{\log_2 7}$
flops. We could show in this paper, that these parameters cause an additional
work of approx.~20 \%  in the worst case in comparison to the optimal strategy
for the worst case. This is the main reason for the search for better
parameters, like in \cite{Probert}.

\end{document}